\documentclass[12pt]{article}
\usepackage[centertags]{amsmath}
\usepackage{amsfonts}
\usepackage{amssymb}
\usepackage{amsthm}
\input{amssym.def}

\numberwithin{equation}{section}

\newtheorem{Definition}{Definition}[section]
\newtheorem{definition}[Definition]{Definition}
\newtheorem{theorem}[Definition]{Theorem}

\newtheorem{proposition}[Definition]{Proposition}
\newtheorem{corollary}[Definition]{Corollary}

\begin{document}
\title{\Large \bf Uniqueness of primary decompositions in Laskerian le-modules}
\author{ A. K. Bhuniya and M. Kumbhakar }
\date{}
\maketitle

\begin{center}
Department of Mathematics, Visva-Bharati, Santiniketan-731235, India. \\
anjankbhuniya@gmail.com, manaskumbhakar@gmail.com
\end{center}

\begin{abstract}{\footnotesize}
Here we introduce and characterize a new class of le-modules $_{R}M$ where  $R$ is a commutative ring with $1$ and $(M,+,\leqslant,e)$ is a lattice ordered semigroup with the greatest element $e$. Several notions are defined and uniqueness theorems for primary decompositions of a submodule element in a Laskerian le-module are established.
\end{abstract}


\section{Introduction}
The observation by W. Krull \cite{Krull} that many results in the ideal theory of commutative rings \cite{Larsen-McCarthy} do not depend on the fact that the ideals are composed of elements, leads to the abstract ideal theory. The researchers chose the system, namely lattice ordered semigroups where an element represents an ideal of the ring as an undivided entity. There are numerous publications dealing with lattice ordered commutative semigroups \cite{Anderson 2, Anderson and Johnson, Fuchs and Reis} and multiplicative lattice \cite{Anderson, Dilworth, Majcherek, Ward and Dilworth 1} -- \cite{Ward and Dilworth 3}, generalizing commutative ideal theory.

What follows, there are numerous publications studying in abstract form the class of submodules of a module over a commutative ring. Researchers considered a complete lattice $M$ together with an action of a multiplicative lattice $L$ on $M$ similar to the action of a ring on an additive abelian group, as if $M$ is the class of all subsets of an $R$-module and $L$ is the class of all subsets of the ring $R$. Such a system $M$ is known as $L$-module or lattice module \cite{Al-Khouja,Johnson 1} -- \cite{Johnson and Johnson, Manjarekar and Bingi, Manjarekar-Kandale, Nakkar and Anderson, Whitman 2, Whitman 1}.

In this article we initiate to develop an ``abstract submodule theory'', analogous to the ``abstract ideal theory'', which would be capable to give insight about rings more directly. Since it is intended that the treatment should be purely submodule theoretic, the system we choose here is a lattice $M$ with a commutative and associative addition. We consider an action of a commutative ring $R$ with $1$ on $M$.

Here we call an le-module $_{R}M$ Laskerian if each submodule element of $M$ is a meet of primary elements of $M$. Two uniqueness theorems for primary decompositions of submodule elements of a Laskerian le-module are proved. Radical of a submodule element $n$ of $M$ is defined to be an ideal of the ring $R$, and, naturally, associated primes of $n$ are prime ideals of $R$. Thus we found direct ways of interaction of a ring $R$ with le-modules.

This article is organized as follows. In the first section we give our definition of an le-module and introduce the notion of submodule elements in le-modules which represent submodules of a module as an undivided entity. Also we prove a number of basic results characterizing submodule elements. Section 3 is devoted to characterize prime and primary submodule elements. We study primary decomposition of a submodule element and their uniqueness in Section 4.

\section{Definition and basic properties of le-modules}
Throughout the article, $\mathbb{N}$ denotes the set of all natural numbers and $R$ stands for a commutative ring with $1$. For the results on rings, modules we refer to \cite{Atiyah-MacDonald, Larsen-McCarthy, CPLu, McCasland and Moore, Smith, Swanson}.

An le-semigroup $(M,+,\leqslant,e)$ is a commutative monoid with the zero element $0_{M}$ at the same time a complete lattice with the greatest element $e$ that satisfies
\begin{enumerate}
\item[(S)] $m+(\vee_{i \in I}m_{i}) = \vee_{i \in I}(m+m_{i})$, for all $m, m_i \in M, i \in I$.
\end{enumerate}

\begin{definition}      \label{order relation}
Let $R$ be a ring and $(M,+,\leqslant)$ be an le-semigroup with the zero element $0_{M}$. Then $M$ is called an le-module over $R$ if there is a mapping $R\times M \longrightarrow M$ satisfying
\begin{enumerate}
\item[(M1)]  $r(m_{1}+m_{2})= rm_{1}+rm_{2}$
\item[(M2)] $(r_{1}+r_{2})m \leqslant r_{1}m+r_{2}m$
\item[(M3)] $(r_{1}r_{2})m=r_{1}(r_{2}m)$
\item[(M4)] $1_{R}.m=m; \;\;\; 0_{R}.m=r.0_{M}=0_{M}$
\item[(M5)] $r(\vee_{i \in I}m_{i}) = \vee_{i \in I}rm_{i}$, for all $r,r_{1},r_{2} \in R$ and $m,m_{1},m_{2},m_{i} \in M$ and $i \in I$.
\end{enumerate}
\end{definition}

We denote an le-module $M$ over $R$ by $_{R}M$. From the Definition \ref{order relation}(M5), we have,
\begin{enumerate}
\item[(M5)$^{\prime}$] $m_{1} \leqslant m_{2} \Rightarrow rm_{1} \leqslant rm_{2}$, for all $r \in R$ and $m_{1},m_{2} \in$ $M$.
\end{enumerate}

Let $_{R}M$ be an le-module. An element $n$ of $M$ is said to be a submodule element if $ n+n , rn \leqslant n$, for all $r \in R$. A submodule element $n$ is called proper if $n \neq e$. Observe that $ 0_{M}= 0_{R}.n \leqslant n$, for every submodule element $n$ of $M$. We denote $(-1).n = -n$. Readers should note that $-n$ is not additive inverse of $n$. Inspite of $0_{R}.m = 0_{M}$, it happens because $(r_{1}+r_{2})m$ and $r_{1}m+r_{2}m$ are not equal, in general. For every submodule element $n$ of $M$, we have $ n = -n $, and so $-n$ is also a submodule element of $M$. Also $ n+n = n $, i.e. every submodule element of $M$  is an idempotent.

Let $_{R}M$ be an le-module and $n$ be an element of $M$. If $A$ is an ideal of $R$, we define
\begin{center}
$ An = \vee \{ \sum_{i = 1}^{k} a_{i}n : k \in \mathbb{N}; a_{1},a_{2},\cdots,a_{k} \in A \}$
\end{center}

\begin{proposition}
Let $_{R}M$ be an le-module. If $A$ is an ideal of $R$ and $n$ is an element of $M$ then $An$ is a submodule element of $M$.
\end{proposition}
\begin{proof}
$An+An = (a_{1}n \vee (a_{1}n+a_{2}n) \vee \cdots) + (b_{1}n \vee (b_{1}n+b_{2}n) \vee \cdots) = ((a_{1}n \vee (a_{1}n+a_{2}n) \vee \cdots)+b_{1}n) \vee ((a_{1}n \vee (a_{1}n+a_{2}n) \vee \cdots)+(b_{1}n+b_{2}n)) \vee \cdots = (a_{1}n+b_{1}n) \vee (a_{1}n+a_{2}n+b_{1}n) \vee \cdots \leqslant An$. Also $r.(An) = r.(a_{1}n \vee (a_{1}n+a_{2}n) \vee \cdots) = (ra_{1})n \vee ((ra_{1})n+(ra_{2})n) \vee \cdots \leqslant An$, since $A$ is an ideal of $R$. Thus $An$ is a submodule element of $_{R}M$.
\end{proof}

\begin{proposition}
Let $_{R}M$ be an le-module, $A$ and $B$ be two ideals of $R$, and $l$ and $n$ be two submodule elements of $M$. Then
\begin{enumerate}
\item[(i)]   $A(Bn) = (AB)n$;
\item[(ii)]  $An \leqslant n$;
\item[(iii)] $A \subseteq B \Rightarrow An \leqslant Bn$;
\item[(iv)]  $l \leqslant n \Rightarrow Al \leqslant Bn$;
\item[(v)]   $A(l+n) = Al+An$;
\item[(vi)]  $(A+B)n = An+Bn$;
\item[(vii)] $A(l \wedge n) \leqslant Al \wedge An$;
\item[(viii)]$(A \cap B)n \leqslant An \wedge Bn$.
\end{enumerate}
\end{proposition}
\begin{proof}
(i)
We have \begin{center}
$ An = \vee \{ \sum_{i = 1}^{k} a_{i}n : k \in \mathbb{N}; a_{1},a_{2},\cdots,a_{k} \in A \}$
\end{center}
For any $a_{1} \in A$, $a_{1}(Bn) = a_{1}(b_{1}n \vee (b_{1}n+b_{2}n) \vee \cdots ) = (a_{1}b_{1})n \vee ((a_{1}b_{1})n+(a_{1}b_{2})n) \vee \cdots \leqslant (AB)n$. For any $a_{1}, a_{2} \in A$, $a_{1}(Bn)+a_{2}(Bn) \leqslant (AB)n+(AB)n \leqslant (AB)n$, since $(AB)n$ is a submodule element of $M$. Thus $A(Bn) = a_{1}(Bn) \vee (a_{1}(Bn)+a_{2}(Bn)) \vee \cdots \leqslant (AB)n \vee (AB)n \vee \cdots = (AB)n$. Therefore $A(Bn) \leqslant (AB)n$. Also $(AB)n = (\sum a_{i}b_{i})n \vee ((\sum c_{i}d_{i})n+(\sum e_{i}f_{i})n) \vee \cdots \leqslant \sum a_{i}(b_{i}n) \vee (\sum c_{i}(d_{i}n)+\sum e_{i}(f_{i}n)) \vee \cdots \leqslant A(Bn) \vee (A(Bn)+A(Bn)) \vee \cdots = A(Bn)$. Thus $(AB)n = A(Bn)$.\\
(ii) For every $k \in \mathbb{N}$ and $a_{1},a_{2}, \cdots, a_{k} \in A$; $a_{1}n+a_{2}n+\cdots+a_{k}n \leqslant n+n+\cdots+n = n$ implies that $An \leqslant n$.\\
(iii) Trivial.\\
(iv) Trivial.\\
(v) From (iv) we have, $Al \leqslant A(l+n)$ and $An \leqslant A(l+n)$. Thus $Al+An \leqslant A(l+n)$, since $A(l+n)$ is a submodule element of $M$. Also for every $k \in \mathbb{N}$ and $a_{1},a_{2}, \cdots, a_{k} \in A$; $a_{1}(l+n)+a_{2}(l+n)+\cdots+a_{k}(l+n) = (a_{1}l+a_{2}l+\cdots+a_{k}l)+(a_{1}n+a_{2}n+\cdots+a_{k}n) \leqslant Al+An$ which implies that $A(l+n) \leqslant Al+An$. Therefore $Al+An = A(l+n)$.\\
(vi) Similar to (v).\\
(vii)Since $l \wedge n \leqslant l$ and $l \wedge n \leqslant n$, so $A(l \wedge n) \leqslant Al$ and $A(l \wedge n) \leqslant An$. Thus $A(l \wedge n) \leqslant Al \wedge An$.\\
(viii) $A \cap B \subseteq A$ and $A \cap B \subseteq B$ together implies that $(A \cap B)n \leqslant An$ and $(A \cap B)n \leqslant Bn$. Thus $(A \cap B)n \leqslant An \wedge Bn$.
\end{proof}

Let $_{R}M$ be an le-module and $n$ be a submodule element of $M$. For any element $r$ of $R$ we set,
\begin{center}
$(n:r) = \vee \{ x \in M : rx \leqslant n \}$.
\end{center}
Then $(n:r)$ is a submodule element of $M$. One can check that $n \leqslant (n:r)$ and $r(n:r) \leqslant n$. For any submodule element of $n$ of $M$ and ideal $A$ of $R$ we set,
\begin{center}
$(n:A) = \vee \{ x \in M: Ax \leqslant n\}$
\end{center}
which is also a submodule element of $M$. Immediately we have following propositions:

\begin{proposition}
Let $_{R}M$ be an le-module, $n$ be a submodule element of $M$ and $A$ be an ideal of $R$. Then
\begin{enumerate}      \label{A(n:A)}
\item[(i)] $n \leqslant (n:A)$,
\item[(ii)] $A(n:A) \leqslant n$.
\end{enumerate}
\end{proposition}
\begin{proof}
(i) We have $An \leqslant n$, and hence $n \leqslant (n:A)$.\\
(ii) Denote $X = \{x \in M : Ax \leqslant n\}$. Then for any $a \in A$, $a(n:A) = a(\vee_{x \in X} x) = \vee_{x \in X}(ax)$ and $ax \leqslant Ax \leqslant n$ implies that $a(n:A) \leqslant n$. Thus for every $k \in \mathbb{N}$ and $a_{1},a_{2}, \cdots, a_{k} \in A$, $a_{1}(n:A)+a_{2}(n:A)+\cdots+a_{k}(n:A) \leqslant n$ and hence $A(n:A) \leqslant n$.
\end{proof}
We omit the proof of the following useful observation, since it follows directly from definition.
\begin{proposition}    \label{Ax lessequal n}
Let $_{R}M$ be an le-module, $n$ be a submodule element of $M$ and $A$ be an ideal of $R$. Then for every $x \in M$ and $r \in R$,
\begin{enumerate}
\item[(i)] $rx \leqslant n $ if and only if $x \leqslant (n:r)$.
\item[(ii)] $Ax \leqslant n $ if and only if $x \leqslant (n:A)$.
\end{enumerate}
\end{proposition}

Let $_{R}M$ be an le-module. If $l$ is a submodule element of $M$ and $n \in M$, we denote
\begin{center}
$(l:n) = \{r \in R: rn \leqslant l\}$
\end{center}
Then $(l:n)$ is an ideal of $R$. Also we have the following propositions.

\begin{proposition}
Let $_{R}M$ be an le-module. Then
\begin{enumerate}
\item[(i)] For every submodule elements $n \leqslant l$ of $_{R}M$, we have $(n : e) \subseteq (l : e)$.
\item[(ii)] If $\{n_{i}\}_{i \in I}$ be a family of submodule elements in $_{R}M$, then
$ (\wedge_{i \in I} n_{i} : e) = \cap_{i \in I} (n_{i} : e)$.
\end{enumerate}
\end{proposition}
We omit the proof, since it is easy.
\begin{proposition}
Let $_{R}M$ be an le-module and $l$ be a submodule element of $M$. Then for any ideal $A$ of $R$ and $n \in M$ , $An \leqslant l $ if and only if $A \subseteq (l:n)$.
\end{proposition}
\begin{proof}
Let $a \in A$. Then $an \leqslant An \leqslant l$ implies that $a \in (l:n)$. Thus $A \subseteq (l:n)$.\\
Conversely, let $A \subseteq (l:n)$. Then for every $k \in \mathbb{N}$ and $a_{1},a_{2}, \cdots, a_{k} \in A$; $a_{1}n+ a_{2}n+\cdots+a_{k}n \leqslant l$, since $l$ is a submodule element. Hence $An \leqslant l$.
\end{proof}

\begin{proposition}    \label{meet}
Let $_{R}M$ be an le-module. Then for every ideals $A$ and $B$ of $R$ and submodule elements $k$, $l$ and $n$ of $M$,
\begin{enumerate}
\item[(i)]   $A \subseteq B \Rightarrow (n:B) \leqslant (n:A)$;
\item[(ii)]  $l \leqslant n \Rightarrow (l:A) \leqslant (n:A)$;
\item[(iii)] $((n:A):B) = (n:AB)$;
\item[(iv)]  $((l \wedge n): A) = (l:A) \wedge (n:A)$;
\item[(v)]   $(n:(A+B)) \leqslant (n:A) \wedge (n:B)$;
\item[(vi)]  $l \leqslant n \Rightarrow (k:n) \subseteq (k:l)$ and $(l:k) \subseteq (n:k)$;
\item[(vii)] $((l \wedge n):k) = (l:k) \cap (n:k)$;
\item[(viii)]$(k:l+n) = (k:l) \cap (k:n)$;
\item[(ix)]  $(k:l) \cap (k:n) = (k:l \vee n)$.
\end{enumerate}
\end{proposition}
\begin{proof}
(i) and (ii) are trivial.\\
(iii) $AB((n:A):B) \leqslant A(n:A) \leqslant n$, (by Proposition \ref{A(n:A)}) implies that $((n:A):B) \leqslant (n:AB)$. Also $AB(n:AB) \leqslant n$ implies that $B(n:AB) \leqslant (n:A)$ and so $(n:AB) \leqslant ((n:A):B)$, by Proposition \ref{Ax lessequal n}. Thus $((n:A):B) = (n:AB)$.\\
(iv) $l \wedge n \leqslant l$ implies that $((l \wedge n):A) \leqslant (l:A)$. Similarly $((l \wedge n):A) \leqslant (n:A)$. Thus $((l \wedge n):A) \leqslant (l:A) \wedge (n:A)$. Also $A((l:A) \wedge (n:A)) \leqslant A(l:A) \wedge A(n:A) \leqslant l \wedge n$. Thus $((l:A) \wedge (n:A)) \leqslant ((l \wedge n):A)$. Therefore $((l \wedge n): A) = (l:A) \wedge (n:A)$.\\
(v) $A \subseteq A+B$ and $B \subseteq A+B$ implies that $(n:(A+B)) \leqslant (n:A)$ and  $(n:(A+B)) \leqslant (n:B)$ and hence $(n:(A+B)) \leqslant (n:A) \wedge (n:B)$. \\
(vi) Let $r \in (k:n)$. Then $rn \leqslant k$ and since $l \leqslant n$, so $rl \leqslant rn \leqslant k$. Thus $r \in (k:l)$. Therefore $(k:n) \subseteq (k:l).$ Now let $r \in (l:k)$. Then $rk \leqslant l \leqslant n$ implies that $r \in (n:k)$. So $(l:k) \subseteq (n:k)$.\\
(vii) $r \in ((l \wedge n):k)$ if and only if $rk \leqslant l \wedge n$ if and only if $rk \leqslant l$ and $rk \leqslant n$ if and only if $r \in (l:k)$ and $r \in (n:k)$ if and only if $r \in (l:k) \cap (n:k)$. Thus $((l \wedge n):k) = (l:k) \cap (n:k)$.\\
(viii) Let $r \in (k:l+n)$. Then $r(l+n) \leqslant k$, i.e, $rl+rn = r(l+n) \leqslant k$. Now $rl = rl+0_{M} \leqslant rl+rn \leqslant k$,(since $rn$ is a submodule element) and $rn = 0_{M}+rn \leqslant rl+rn \leqslant k$ shows that $r \in (k:l)$ and $r \in (k:n)$. So $r \in (k:l) \cap (k:n)$. Thus $(k:l+n) \subseteq (k:l) \cap (k:n)$. Now let $r \in (k:l) \cap (k:n)$. Then $rl \leqslant k$ and $rn \leqslant k$. Thus $r(l+n) = rl+rn \leqslant k+k = k$,(since $k$ is a submodule element), i.e, $r \in (k:l+n)$. Therefore $(k:l) \cap (k:n) \subseteq (k:l+n)$. Hence $(k:l+n) = (k:l) \cap (k:n)$.\\
(ix) Let $r \in (k:l \vee n)$. Then $rl \vee rn = r(l \vee n) \leqslant k$. Now $rl \leqslant rl \vee rn \leqslant k$ implies that $r \in (k:l)$. Similarly $r \in (k:n)$. Therefore $r \in (k:l) \cap (k:n)$. Thus $(k:l \vee n) \subseteq (k:l) \cap (k:n)$. Now let $r \in (k:l) \cap (k:n)$. Then $rl \leqslant k$ and $rn \leqslant k$. It follows that $rl \vee rn \leqslant k$, i.e, $r(l \vee n) \leqslant k$. So $r \in (k:l \vee n)$. Thus $(k:l) \cap (k:n) \subseteq (k:l \vee n)$. Hence $(k:l) \cap (k:n) = (k:l \vee n)$
\end{proof}

\section{Prime and primary elements}
In this section we introduce and characterize the prime and primary elements in an le-module.
\begin{definition}
Let $_{R}M$ be an le-module and $n$ be a proper submodule element of $M$. Then $n$ is called a primary element of $M$ if for any $a \in R$ and $x \in M$, $ax \leqslant n$ implies that either $x \leqslant n$ or $a^me \leqslant n$, i.e, either $x \leqslant n$ or $a^m \in (n:e)$, for some $m \in \mathbb{N}$.
\end{definition}

Recall that if $I$ is an ideal of a ring $R$, then
\begin{center}
Rad$(I) = \{a \in R : a^n \in I$, for some positive integer $n$\}
\end{center}
and we call Rad$(I)$ the radical of the ideal $I$ in $R$.

Let $_{R}M$ be an le-module and $n$ be a submodule element of $M$. Then $(n:e)$ is an ideal of $R$. Now we set
\begin{center}
Rad$(n)$ = Rad$(n:e)$,
\end{center}
and we call Rad$(n)$ the radical of the submodule element $n$. Now we have the following propositions:

\begin{proposition}    \label{Radical of q}
If $l,n$ and $n_{1},n_{2},\cdots,n_{k}$ are submodule elements of an le-module $_{R}M$ then
\begin{enumerate}
\item[(i)] $l \leqslant n \Rightarrow Rad(l) \subseteq Rad(n)$
\item[(ii)] $Rad(Rad(n)) = Rad(n)$
\item[(iii)] $Rad(\wedge n_{i}) = Rad(n_{1}) \cap Rad(n_{2}) \cap \cdots \cap Rad(n_{k})$.
\end{enumerate}
\end{proposition}
\begin{proof}
(i) Let $a \in Rad(l)$. Then $a^me \leqslant l \leqslant n$, for some $m \in \mathbb{N}$. Therefore $a \in Rad(n)$ and so $Rad(l) \subseteq Rad(n)$.\\
(ii) We have $Rad(n) \subseteq Rad(Rad(n))$. Let $a \in Rad(Rad(n))$. Then $a^m \in Rad(n)$, for some $m \in \mathbb{N}$. This implies that $(a^m)^ke \leqslant n$, for some $k \in \mathbb{N}$, i.e, $a^{t}e \leqslant n$, for some $t = mk \in \mathbb{N}$. Thus $a \in Rad(n)$. Hence $Rad(Rad(n)) = Rad(n)$.\\
(iii) Let $a \in Rad(\wedge n_{i})$. Then $a^me \leqslant \wedge n_{i} \leqslant n_{i}$, for some $m \in \mathbb{N}$. This shows that $a \in Rad(n_{i})$, for each $i$ and hence $a \in Rad(n_{1}) \cap Rad(n_{2}) \cap \cdots \cap Rad(n_{k})$. Thus $Rad(\wedge n_{i}) \subseteq  Rad(n_{1}) \cap Rad(n_{2}) \cap \cdots \cap Rad(n_{k})$. Now let $a \in Rad(n_{1}) \cap Rad(n_{2}) \cap \cdots \cap Rad(n_{k})$. Then $a^{m_{1}}e \leqslant n_{1}, a^{m_{2}}e \leqslant n_{2}, \cdots, a^{m_{k}}e \leqslant n_{k}$, for some $m_{1},m_{2},\cdots,m_{k} \in \mathbb{N}$. Let $r = max\{m_{1},m_{2},\cdots,m_{k}\}$. Then $a^re \leqslant n_{i}$, for each $i$, i.e, $a^re \leqslant \wedge n_{i}$. Therefore $a \in Rad(\wedge n_{i})$ and it follows that $Rad(n_{1}) \cap Rad(n_{2}) \cap \cdots \cap Rad(n_{k}) \subseteq Rad(\wedge n_{i})$. Hence $Rad(\wedge n_{i}) = Rad(n_{1}) \cap Rad(n_{2}) \cap \cdots \cap Rad(n_{k})$.
\end{proof}

\begin{proposition}
Let $q$ be a primary element of an le-module $_{R}M$. Then Rad$(q)$ is a prime ideal of $R$.
\end{proposition}
\begin{proof}
Let $a,b \in$ Rad$(q)$. Then $a^me \leqslant q$ and $b^ne \leqslant q$, for some $m, n \in \mathbb{N}$. So
\begin{center}
$(a-b)^{m+n}e = \displaystyle \sum_{k=0}^{m+n} (-1)^k {m+n \choose k}a^{m+n-k}b^ke \leqslant q$
\end{center}
since either $m+n-k \geqslant m$ or $k \geqslant n$, and $q$ is a submodule element. Thus $a-b \in$ Rad$(q)$. Also for any $a \in$ Rad$(q)$ and $r \in R$, $(ra)^me = r^ma^me \leqslant q$ shows that $ra \in$ Rad$(q)$. Hence Rad$(q)$ is an ideal of $R$. Let $ab \in$ Rad$(q)$ and $b \notin$ Rad$(q)$. Then $(ab)^n = a^nb^n \in (q:e)$, for some $n \in \mathbb{N}$, i.e, $a^nb^ne \leqslant q$. Since $b \notin$ Rad$(q)$ so $b^r \notin (q:e)$, for every $r \in \mathbb{N}$. In particular $b^n \notin (q:e)$, i.e, $b^ne \nleqslant q$. But $q$ is a primary element and so $(a^n)^ke \leqslant q$, for some $k \in \mathbb{N}$. Thus $a \in$ Rad$(q)$ and hence Rad$(q)$ is a prime ideal of $R$.
\end{proof}

If $q$ is a primary element of an le-module $_{R}M$ and Rad$(q) = P$, we say that $q$ is $P$-primary.

Our following proposition gives a relation between primary elements of an le-module $_{R}M$ and maximal ideals of $R$.
\begin{proposition}
Let $q$ be a submodule element of an le-module $_{R}M$. If $Rad(q) = m$ is a maximal ideal of $R$ then $q$ is a primary element.
\end{proposition}
\begin{proof}
Let $a \in R$ and $x \in M$ be such that $ax \leqslant q$ and $x \nleqslant q$. If possible, let $a \notin Rad(q)$. Then $\langle Rad(q)\cup \{a\}\rangle = R$, by the maximality of $Rad(q)$. Since $1 \in R$, $1 = u+ra$, for some $u \in Rad(q)$ and $r \in R$. Since $u \in Rad(q)$ so $u^me \leqslant q$, for some $m \in \mathbb{N}$. Now $1 = 1^m = (u+ra)^m = u^m+sa$, for some $s \in R$. Then $x = u^mx+sax \leqslant u^me+sq \leqslant q+q = q$, a contradiction. Thus $q$ is a $m$-primary element.
\end{proof}

\begin{proposition}
Let $_{R}M$ be an le-module and $q_{1},q_{2},\cdots,q_{r}$ be $P$-primary elements of $M$. Then $q_{1} \wedge q_{2} \wedge \cdots \wedge q_{r}$ is $P$-primary.
\end{proposition}
\begin{proof}
Let $q = q_{1} \wedge q_{2} \wedge \cdots \wedge q_{r}$. It is evident that $Rad(\wedge q_{i}) = Rad(q_{1}) \cap Rad(q_{2}) \cap \cdots \cap Rad(q_{r}) = P$, by Proposition \ref{Radical of q}.
Now suppose that $ax \leqslant q$, $a \in R$, $x \in M$. If $x \nleqslant q$, then $x \nleqslant q_{i}$, for some $i(1 \leqslant i \leqslant r)$. Since $q_{i}$ is primary, $ax \leqslant q \leqslant q_{i}$ and $x \nleqslant q_{i}$ implies that $a \in Rad(q_{i}) = P = Rad(q)$. Thus $q$ is $P$-primary.
\end{proof}

\begin{theorem}    \label{(q:n) is primary}
Let $q$ be a $P$-primary element of an le-module $_{R}M$ and $n \in M$. Then the following results hold:
\begin{enumerate}
\item[(i)] If $n \leqslant q$ then $(q:n) = R$.
\item[(ii)] If $n \nleqslant q$ then $(q:n)$ is a $P$-primary ideal of $R$.
\end{enumerate}
\end{theorem}
\begin{proof}
(i) Let $n \leqslant q$ and $r \in R$. Then $rn \leqslant rq \leqslant q$ implies that $r \in (q:n)$ for every $r \in R$. Therefore $(q:n) = R$.\\
(ii)We have $(q:n)$ is an ideal of $R$. Let $a,b \in R$ be such that $ab \in (q:n)$ and $b \notin (q:n)$. Then $abn \leqslant q$ and $bn \nleqslant q$. Since $q$ is primary, $a^me \leqslant q$, for some $m \in \mathbb{N}$. Now $a^mn \leqslant a^me \leqslant q$ implies that $a^m \in (q:n)$. Thus $(q:n)$ is a primary ideal of $R$. To show $Rad(q:n) = P$, let $a \in Rad(q:n)$. Then $a^mn \leqslant q$, for some $m \in \mathbb{N}$. Since $q$ is primary and $n \nleqslant q$, so $(a^m)^ke \leqslant q$, for some $k \in \mathbb{N}$. Thus $a \in Rad(q)$, and hence $Rad(q:n) \subseteq P$. Now let $a \in P$. Then $a^m \in (q:e)$, for some $m \in \mathbb{N}$ implies that $a^mn \leqslant a^me \leqslant q$. Hence $a \in Rad(q:n)$ and so $P \subseteq Rad(q:n)$. Therefore $Rad(q:n) = P$. Hence $(q:n)$ is a $P$-primary ideal of $R$.
\end{proof}

\begin{theorem}
Let $q$ be a $P$-primary element of an le-module $_{R}M$ and $a \in R$. If $a \notin P$ then $(q:a) = q$.
\end{theorem}
\begin{proof}
We have  $a(q:a) = a(\vee_{ax \leqslant q} x) = \vee_{ax \leqslant q} (ax) \leqslant q$. Since $q$ is $P$-primary and $a \notin P$, so it follows that $(q:a) \leqslant q$. Also we have $q \leqslant (q:a)$. Thus $(q:a) = q$.
\end{proof}

Now we introduce the prime submodule elements of an le-module $_{R}M$.
\begin{definition}
A proper submodule element $p$ of an le-module $_{R}M$ is said to be a prime submodule element if for every $r \in R$ and $n \in M$, $rn \leqslant p$ implies that $r \in (p:e) $ or $n \leqslant p$.
\end{definition}

In the following result we characterize relationship of the prime submodule elements of $M$ with the prime ideals of the ring $R$.
\begin{theorem}
Let $p$ be a prime submodule element of an le-module $_{R}M$ and $x \in M$. Then $(p:x)$ is a prime ideal of $R$ for every $x \in M$.
\end{theorem}
\begin{proof}
We have $0 \in (p:x)$ and so $(p:x) \neq \emptyset$. Let $a,b \in (p:x)$. Then $ax \leqslant p$ and $ bx \leqslant p $ implies that $(a-b)x \leqslant ax+(-b)x \leqslant p+(-p) \leqslant p+p \leqslant p$ and so $ a-b \in (p:x)$. Also $(ra)x = r(ax) \leqslant rp \leqslant p$ implies that $ra \in (p:x)$. Thus $(p:x)$ is an ideal of $R$. Let $a,b \in R$ be such that $ab \in (p:x)$. Then $ (ab)x \leqslant p$, i.e. $a(bx) \leqslant p$. Since $p$ is a prime submodule element, so  $ a \in (p:e) \subseteq (p:x)$ or $ bx \leqslant p$, i.e. $ a \in (p:x) $ or $ b \in (p:x) $. Hence $(p:x)$ is a prime ideal of $R$.
\end{proof}
In particular we have:
\begin{corollary}     \label{prime ideal}
If $p$ is a prime submodule element of an le-module $_{R}M$, then $(p:e)$ is a prime ideal of $R$
\end{corollary}

The converse of the above corollary does not hold, in general. In the next theorem we prove that if, in addition, $p$ is assumed to be primary then the converse holds. This result also characterizes when a primary element becomes a prime submodule element.
\begin{theorem}     \label{prime submodule element}
Let $p$ be a proper submodule element of an le-module $_{R}M$. If $(p:e)$ is a prime ideal of $R$ and $p$ is primary, then $p$ is a prime submodule element.
\end{theorem}
\begin{proof}
Let $rn \leqslant p$ and $n \nleqslant p$ for $r \in R$ and $n \in M$. Since $p$ is primary so $r^me \leqslant p$, for some $m \in \mathbb{N}$, i.e, $r^m \in (p:e)$. Since $(p:e)$ is a prime ideal of $R$ so $r \in (p:e)$. Consequently, $p$ is a prime submodule element.
\end{proof}

Also we have another characterization for the converse of Corollary \ref{prime ideal}.
\begin{proposition}
Let $p$ be a submodule element of an le-module $_{R}M$. If $(p:e)$ is a maximal ideal of $R$ then $p$ is a prime submodule element of $M$.
\end{proposition}
\begin{proof}
Let $a \in R$ and $n \in M$ be such that $an \leqslant p$. If possible, let $a \notin (p:e)$, i.e, $ae \nleqslant p$. Then $\langle (p:e)\cup \{a\}\rangle = R$, since $(p:e)$ is maximal. Since $1 \in R$, $1 = r+sa$, for some $r \in (p:e)$ and $s \in R$. Also $rn \leqslant re \leqslant p$, since $r \in (p:e)$. Now $n = (r+sa)n \leqslant rn+san \leqslant p+p = p$. Hence $p$ is a prime submodule element.
\end{proof}

\section{Primary decomposition in Laskerian le-module}
An $R$-module $M$ is called Laskerian if each submodule is a finite intersection of primary submodules. Many authors assumed Laskerian modules to be finitely generated. For further information on Laskerian modules readers are referred to \cite{Gilmer and Heinzer, Heinzer and Lantz, Heinzer and Ohm, Jayaram and Tekir, Radu, Visweswaran}.

Let $_{R}M$ be an le-module and $n$ be a submodule element of $M$. Then $n$ is said to have a primary decomposition if there exist primary elements $q_{1},q_{2},\cdots,q_{k}$ of $M$ such that
\begin{center}
$n = q_{1} \wedge q_{2} \wedge \cdots \wedge q_{k}$.
\end{center}

A primary decomposition is called \emph{reduced} if:
\begin{enumerate}
\item[(i)] $q_{1} \wedge q_{2} \wedge \cdots \wedge q_{i-1} \wedge q_{i+1} \wedge \cdots \wedge q_{k} \nleqslant q_{i}$, for all $i = 1,2, \cdots, k$.
\item[(ii)] Rad$(q_{i}) \neq Rad(q_{j})$ for $i \neq j$.
\end{enumerate}

An le-module $_{R}M$ is said to be \emph{Laskerian} if every submodule element of $_{R}M$ has a primary decomposition.

Throughout the rest of the article, every le-module $_{R}M$ is Laskerian.
It is easy to observe that if a submodule element $n$ of an le-module $_{R}M$ has a primary decomposition then it has a reduced primary decomposition.

\begin{theorem}
Let $n$ be a proper submodule element of an le-module $M$ which has a reduced primary decomposition $n = q_{1} \wedge q_{2}\wedge \cdots \wedge q_{k}$. Then every $q_{i}$ is a prime submodule element if and only if $Rad(n) = (n:e)$.
\end{theorem}
\begin{proof}
First assume that each $q_{i}$ is a prime submodule element. We have $(n:e) \subseteq Rad(n)$. Now let $a \in Rad(n)$. Then $a^me \leqslant n =\wedge q_{i}$, for some $m \in \mathbb{N}$. Therefore $a^me \leqslant q_{i}$, for each $i$. Since $q_{i}$ is prime so $ae \leqslant q_{i}$ or $a^{m-1}e \leqslant q_{i}$, for each $i$. If $ae \leqslant q_{i}$ then $a \in (q:e)$. Otherwise $a^{m-1} \leqslant q_{i}$ implies $ae \leqslant q_{i}$ or $a^{m-2}e \leqslant q_{i}$. Continuing in this way we get $a \in (q_{i}:e)$, for each $i$. Thus $a \in (q_{1}:e) \cap (q_{2}:e) \cap \cdots \cap (q_{k}:e)$, i.e, $a \in (n:e)$. Hence $Rad(n) = (n:e)$.

Conversely assume that $Rad(n) = (n:e)$. Let $Rad(q_{i}) = P_{i}$. We claim that $Rad(q_{i}) = (q_{i}:e)$. Let $a \in P_{i}$. Since $\cap_{i=1}^{k}P_{i}$ is reduced there exists $b \in \cap P_{j}$ but $b \notin P_{i}$ in $R$. Now $ab \in \cap_{i=1}^{k}P_{i} = \cap_{i=1}^{k}Rad(q_{i}) = Rad(n) = (n:e) = \cap_{i=1}^{k}(q_{i}:e)$ implies $abe \leqslant q_{i}$, for each $i$. Since $q_{i}$ is primary and $b \notin P_{i}$, so $ae \leqslant q_{i}$, i.e, $a \in (q_{i}:e)$. Consequently, $P_{i} = (q_{i}:e)$. Hence by Theorem \ref{prime submodule element}, each $q_{i}$ is a prime submodule element.
\end{proof}

Let $I$ be an ideal of a ring $R$. Recall that a prime ideal $P$ of $R$ is called a \emph{minimal prime divisor} of $I$ if $I \subseteq P$ and if there is no prime ideal $P'$ of $R$ such that $I \subset P' \subset P$.

Let $n$ be a submodule element of $M$. Then the minimal prime divisors of $(n:e)$ are called \emph{minimal prime divisors} of $n$. Immediately we have
\begin{center}
Rad$(n) = \cap P$
\end{center}
where the intersection is over all minimal prime divisors of $n$. We omit the proof since it is similar to the analogous result in the ideal theory of rings.

Let $n = q_{1} \wedge q_{2} \wedge \cdots \wedge q_{m}$ be a reduced primary decomposition of a submodule element $n$ of $_{R}M$ and let $P_{i} = Rad(q_{i})$, for $i = 1,2,\cdots,m$. The prime ideals $P_{1},P_{2},\cdots,P_{m}$ are called \emph{prime divisors} of $n$ or \emph{associated primes} of $n$.
\begin{theorem}
Let $n$ be a proper submodule element of an le-module $M$ which has a reduced primary decomposition $n = q_{1} \wedge q_{2}\wedge \cdots \wedge q_{k}$ and $P_{i} = Rad(q_{i})$ be the associated prime ideals of $n$. Let $P$ be a prime ideal of $R$ then $(n:e) \subseteq P$ if and only if $P_{i} \subseteq P$ for some $i = 1,2, \cdots ,k$.
\end{theorem}
\begin{proof}
Let  $(n:e) \subseteq P$. Then, by Proposition \ref{meet}, $\cap_{i=1}^{k}(q_{i}:e) \subseteq P$ which implies that $(q_{i}:e) \subseteq P$, for some $i$. Now $P_{i}$ is the smallest prime ideal containing $(q_{i}:e)$, hence $P_{i} \subseteq P$, for some $i$. Conversely suppose $P_{j} \subseteq P$, for some $j$. Then $(n:e) = \cap_{i=1}^{k}(q_{i}:e) \subseteq \cap_{i=1}^{k}Rad(q_{i}) \subseteq Rad(q_{j}) = P_{j} \subseteq P$.
\end{proof}
\begin{corollary}
Let $n$ be a submodule element of an le-module $M$. Then every minimal prime divisor of $n$ is a prime divisor of $n$ and is minimal in the set of prime divisors of $n$.
\end{corollary}

\begin{theorem} [1st Uniqueness Theorem]
Let $n$ be a submodule element of an le-module $_{R}M$ and assume that $n$ has a reduced primary decomposition $n = q_{1} \wedge q_{2} \wedge \cdots \wedge q_{r}$. Let $P$ be a prime ideal of $R$. Then $P = Rad(q_{i})$, for some $i$ if and only if $(n:x)$ is a $P$-primary ideal of $R$, for some $x \nleqslant n$. Hence the set of all associated primes is independent of primary decomposition of $n$.
\end{theorem}
\begin{proof}
Let $P_{i} = Rad(q_{i})$, for $i=1,2,\cdots,m$. Without loss of generality, we assume that $P = P_{1}$. Since the decomposition is reduced, in particular $q_{2} \wedge q_{3} \wedge \cdots \wedge q_{m} \nleqslant q_{1}$. Let $x = q_{2} \wedge q_{3} \wedge \cdots \wedge q_{m}$. Then $x \nleqslant n$ and $(n:x) = (q_{1} \wedge q_{2} \wedge \cdots \wedge q_{m}:x) = (q_{1}:x) \cap (q_{2}:x) \cap \cdots \cap (q_{m}:x)$. By Theorem \ref{(q:n) is primary}, $(q_{i}:x) = R$, for $i=2,3,\cdots,m$ and whence $(n:x) = (q_{1}:x)$. Thus $(n:x)$ is a $P$-primary ideal of $R$, by Theorem \ref{(q:n) is primary}.

Conversely suppose that $(n:x)$ is a $P$-primary ideal of $R$, for some $x \nleqslant n$. Now we have $P = Rad(n:x) = Rad(q_{1}:x) \cap Rad(q_{2}:x) \cap \cdots \cap Rad(q_{m}:x)$. Now $x \nleqslant n$ implies that $x \nleqslant q_{i}$, for some $i$ say $x \nleqslant q_{k}$; $1 \leqslant k \leqslant m$. Then by the fact that $(q_{i}:x) \subseteq Rad(q_{i}:x)$ and by Theorem \ref{(q:n) is primary}, we have, for each $i$, $Rad(q_{i}:x) = P_{i}$ or $R$, and is equal to $P_{i}$ for $i = k$. Hence $P$ is the intersection of some prime ideals $P_{i_{1}},P_{i_{2}},\cdots,P_{i_{h}}$ ($1 \leqslant h \leqslant m$).i.e, $P = P_{i_{1}} \cap P_{i_{2}} \cap \cdots \cap P_{i_{h}}$. Thus $P = P_{i}$, for some $i$.
\end{proof}

The following consequence is immediate.
\begin{corollary}
Let $n$ be a submodule element of an le-module $_{R}M$ and assume that $n$ has a primary decomposition. If
\begin{center}
$n = q_{1} \wedge q_{2} \wedge \cdots \wedge q_{r} = q_{1}' \wedge q_{2}' \wedge \cdots \wedge q_{s}'$
\end{center}
are two reduced primary decompositions of $n$, then $r = s$ and the $q_{i}$ and $q_{i}'$ can be so numbered that $Rad(q_{i}) = Rad (q_{i}')$ for $i = 1,2,\cdots,r$.
\end{corollary}

Let $S$ be a multiplicatively closed set in $R$ and $n$ be a submodule element of an le-module $M$. We define
\begin{center}
$n_{S} = \vee \{x \in M: sx \leqslant n$, for some $s \in S\}$.
\end{center}
Note that if $(n:e) \cap S \neq \emptyset$ then $n_{S} = e$. Thus if $0 \in S$, then $n_{S} = e$.
If $n$ is a submodule element of $M$ then $n_{S}$ is a submodule element of $M$. We call the submodule element $n_{S}$ the \emph{$S$-component} of $n$.

\begin{proposition}     \label{S component}
Let $_{R}M$ be an le-module and $n$ be a submodule element $M$ having a primary decomposition $n = q_{1} \wedge q_{2} \wedge \cdots \wedge q_{m}$ where $q_{i}$ are $P_{i}$-primary. If $S$ is a multiplicatively closed set in $R$ such that $P_{i} \cap S = \emptyset$, for $i = 1,2,\cdots,k$ and $P_{i} \cap S \neq \emptyset$, for the remaining $i$ then $n_{S} = q_{1} \wedge q_{2} \wedge \cdots \wedge q_{k}$.
\end{proposition}
\begin{proof}
Let $T = \{x \in M: sx \leqslant n$, for some $s \in S\}$. Then $n_{S} = \vee T$. Let $x \in T$. Then there exists $s \in S$ such that $sx \leqslant n = \wedge_{i=1}^{m} q_{i} \leqslant q_{i}$, for $i = 1,2,\cdots,m$. Since $P_{i} \cap S = \emptyset$, $s \notin P_{i}$ and hence $x \leqslant q_{i}$, for $i=1,2,\cdots,k$. It follows that $n_{S} = \vee T \leqslant q_{1} \wedge q_{2} \wedge \cdots \wedge q_{k}$.

Now denote $q = q_{1} \wedge q_{2} \wedge \cdots \wedge q_{k}$. If $k = m$, then $q = n \in T$ implies that $q_{1} \wedge q_{2} \wedge \cdots \wedge q_{k} = q \leqslant \vee T = n_{S}$. Thus $n_{S} = q_{1} \wedge q_{2} \wedge \cdots \wedge q_{k}$. If $k \neq m$, then for $j=k+1,k+2,\cdots,m$ choose $s_{j} \in P_{j} \cap S$ and so $s_{j}^{r_{j}} e \leqslant q_{j}$ for some $r_{j} \in \mathbb{N}$. Then for large enough $r \in \mathbb{N}$, we have
\begin{center}
$(s_{k+1} \cdot s_{k+2} \cdots s_{m})^r e \leqslant q_{k+1} \wedge q_{k+2} \wedge \cdots \wedge q_{m}$
\end{center}
and hence $(s_{k+1} \cdot s_{k+2} \cdots s_{m})^r q \leqslant (s_{k+1} \cdot s_{k+2} \cdots s_{m})^r e \leqslant q_{k+1} \wedge q_{k+2} \wedge \cdots \wedge q_{m}$. Also $(s_{k+1} \cdot s_{k+2} \cdots s_{m})^r q = (s_{k+1} \cdot s_{k+2} \cdots s_{m})^r (q_{1} \wedge q_{2} \wedge \cdots \wedge q_{k}) \leqslant (s_{k+1} \cdot s_{k+2} \cdots s_{m})^rq_{1} \wedge (s_{k+1} \cdot s_{k+2} \cdots s_{m})^rq_{2} \wedge \cdots \wedge (s_{k+1} \cdot s_{k+2} \cdots s_{m})^rq_{k} \leqslant q_{1} \wedge q_{2} \wedge \cdots \wedge q_{k}$, since $S$ is multiplicatively closed and $q_{i}$'s are submodule elements. Thus $(s_{k+1} \cdot s_{k+2} \cdots s_{m})^r q \leqslant q_{1} \wedge q_{2} \wedge \cdots \wedge q_{m} = n$ which implies that $q \in T$ and so $q \leqslant n_{S}$. Therefore $q_{1} \wedge q_{2} \wedge \cdots \wedge q_{k} \leqslant n_{S}$. Hence $n_{S} = q_{1} \wedge q_{2} \wedge \cdots \wedge q_{k}$.
\end{proof}

Let $_{R}M$ be an le-module and $q$ be a submodule element of $M$ which  has a reduced primary decomposition $q = q_{1} \wedge q_{2} \wedge \cdots \wedge q_{n}$ and $P_{i} = Rad(q_{i})$, for $i=1,2,\cdots,n$. Then $P_{i}$ is called an \emph{isolated prime divisor} of $n$ if $P_{i}$ is minimal in the set of all prime divisors of $n$, i.e, $P_{i}$ does not contain properly any $P_{j}$, ($i \neq j$). If $P_{i}$ is an isolated prime divisor of $n$ then $q_{i}$ is called an \emph{isolated component} of $n$. In the next theorem we show that the isolated components of $n$ are uniquely determined by $n$.

\begin{theorem}
Let $_{R}M$ be an le-module and $n$ be a proper submodule element of $M$ which has a reduced primary decomposition $n = q_{1} \wedge q_{2} \wedge \cdots \wedge q_{m}$ and $P_{i} = Rad(q_{i})$. Then
\begin{center}
$q_{i}' = \vee \{x \in M : (n:x) \nsubseteq P_{i}\}$
\end{center}
is a submodule element of $M$ which is contained in $q_{i}$. If $q_{i}$ is an isolated primary component of $n$ then $q_{i}= q_{i}'$.
\end{theorem}

\begin{proof}
Let $T = \{x \in M : (n:x) \nsubseteq P_{i}\}$. Consider $x, y \in T$ and $r \in R$. Then $(n:x) \nsubseteq P_{i}$ and $(n:y) \nsubseteq P_{i}$ and so there exist $a, b \in R$ such that $ax \leqslant n$ and $by \leqslant n$ but $a, b \notin P_{i}$. Therefore $ab(x+y) = (ab)x+(ab)y \leqslant bn+an \leqslant n+n = n$ implies that $ab \in (n:x+y)$. Since $P_{i}$ is a prime ideal and $a, b \notin P_{i}$, $ab \notin P_{i}$. Thus $(n:x+y) \nsubseteq P_{i}$ and hence $x+y \in T$. Also $arx = rax \leqslant rn \leqslant n$ implies that $a \in (n:rx)$. But $a \notin P_{i}$ and so $(n:rx) \nsubseteq P_{i}$. Thus $rx \in T$. Hence $q_{i}' + q_{i}' = \vee_{x \in T} x + \vee_{y \in T} y = \vee_{x,y \in T}(x+y) \leqslant q_{i}'$ and $rq_{i}' = r(\vee_{x \in T} x) = \vee_{x \in T}(rx) \leqslant q_{i}'$ implies that $q_{i}'$ is a submodule element of $M$. Let $x \in T$. Then $(n:x) \nsubseteq P_{i}$, and so there exists $a \in (n:x)$ such that $a \notin P_{i}$. Thus $ax \leqslant n \leqslant q_{i}$ which implies that $x \leqslant q_{i}$, since $q_{i}$ is a primary element. Hence $q_{i}' = \vee_{x \in T} x \leqslant q_{i}$.

If $q_{i}$ is an isolated primary component of $n$ then $Rad(q_{i}) = P_{i}$ is a minimal associated prime of $n$ and so $P_{j} \nsubseteq P_{i}$ for $i \neq j$. Then there exists $a_{j} \in P_{j}$ such that $a_{j} \notin P_{i}$. Thus $a_{j}^{r_{j}}e \leqslant q_{j}$ for some $r_{j} \in \mathbb{N}$. Let $a = \displaystyle\prod_{j \neq i} a_{j}^{r_{j}}$. Then $ae \leqslant \wedge_{j \neq i} q_{j}$ but $a \notin P_{i}$, since $P_{i}$ is prime ideal. Now $aq_{i} \leqslant ae \leqslant \wedge_{j \neq i} q_{j}$ and $aq_{i} \leqslant q_{i}$ implies that $aq_{i} \leqslant \wedge_{i = 1}^{m} q_{i} = n$ and hence $a \in (n:q_{i})$. Thus $(n:q_{i}) \nsubseteq P_{i}$ and so $q_{i} \in \{x \in M : (n:x) \nsubseteq P_{i}\}$. Hence $q_{i} \leqslant q_{i}'$ and it follows that $q_{i} = q_{i}'$.
\end{proof}

\begin{theorem}      \label{isolated prime divisor}
Let $_{R}M$ be an le-module and $n$ be a proper submodule element of $M$ which has a reduced primary decomposition $n = q_{1} \wedge q_{2} \wedge \cdots \wedge q_{m}$. Then $Rad(n)$ is the intersection of isolated prime divisors of $n$.
\end{theorem}
\begin{proof}
We have $Rad(n) = Rad(q_{1} \wedge q_{2} \wedge \cdots \wedge q_{m}) = Rad(q_{1}) \cap Rad(q_{2}) \cap \cdots \cap Rad(q_{m}) = P_{1} \cap P_{2} \cap \cdots \cap P_{m}$, where $P_{i} = Rad(q_{i})$ are prime divisors of $n$. If some $P_{j}$ is not isolated then $P_{i} \subseteq P_{j}$ for some $P_{i}$ and hence we can delete such $P_{j}$ from the above and we are done.
\end{proof}

\begin{theorem}
Let $_{R}M$ be an le-module and $n$ be a submodule element of $M$ which has a primary decomposition. Then $Rad(n)$ is a prime ideal of $R$ if and only if $n$ has a single isolated prime divisor.
\end{theorem}
\begin{proof}
Let $n = q_{1} \wedge q_{2} \wedge \cdots \wedge q_{m}$ be a primary decomposition of $n$. Assume that $Rad(n)$ is a prime ideal of $R$ and if possible suppose that $n$ has two isolated prime divisors $P_{i_{1}}$ and $P_{i_{2}}$. Then, by Theorem \ref{isolated prime divisor}, $Rad(n) = P_{i_{1}} \cap P_{i_{2}}$. Since $P_{i_{1}}$ and $P_{i_{2}}$ are isolated prime divisors of $n$ so there exist $a \in P_{i_{1}}, a \notin P_{i_{2}}$ and $b \in P_{i_{2}}, b \notin P_{i_{1}}$. Hence $ab \in P_{i_{1}} \cap P_{i_{2}} = Rad(n)$. Since $Rad(n)$ is prime ideal so $a \in Rad(n)$ or $b \in Rad(n)$ which implies that $a \in P_{i_{2}}$ or $b \in P_{i_{1}}$, a contradiction. Thus $n$ has a single isolated prime divisor. Similar is the case when $n$ has more than two isolated prime divisors.

Converse part is obvious.
\end{proof}

\begin{theorem}
Let $_{R}M$ be an le-module, $n$ be a proper submodule element of $M$ and $r \in R$. Then $n = (n:r)$ if and only if $r$ is contained in no associated prime divisor of $n$.
\end{theorem}
\begin{proof}
Let $n = q_{1} \wedge q_{2} \wedge \cdots \wedge q_{m}$ be a reduced primary decomposition of $n$ and $Rad(q_{i}) = P_{i}$. Suppose $r \notin P_{i}$ for every $i =1,2, \cdots ,m$. We have $r(n:r) \leqslant n$ and so $r(n:r) \leqslant q_{i}$ for all $i$. Since $q_{i}$ is primary and $r \notin P_{i}$, so $(n:r) \leqslant q_{i}$ for all $i$. Thus $(n:r) \leqslant n$. Also $n \leqslant (n:r)$. Hence $(n:r) = n$.

Conversely let $(n:r) = n$. If possible, without any loss of generality, assume that $r \in P_{1}$. Then $r^ke \leqslant q_{1}$ for some $k \in \mathbb{N}$ which implies that $(q_{1}:r^k) = e$. One can easily check that $((n:r):r) = (n:r^2)$ and so $(n:r) = (n:r^k)$. Thus $(n:r^k) = n$. Now $n = (n:r^k) = (q_{1} \wedge q_{2} \wedge \cdots \wedge q_{m}:r^k) = (q_{1}:r^k) \wedge (q_{2}:r^k) \wedge \cdots \wedge (q_{m}:r^k) = \wedge_{i\neq 1}(q_{i}:r^k) \geqslant \wedge_{i \neq 1}q_{i} \geqslant n$. Thus $n = \wedge_{i \neq 1}q_{i}$ which contradicts that $n = q_{1} \wedge q_{2} \wedge \cdots \wedge q_{m}$ is a reduced primary decomposition of $n$.
\end{proof}

\begin{theorem}[2nd Uniqueness Theorem]     \label{isolated set of prime divisors}
Let $_{R}M$ be an le-module and $n$ be a submodule element of $M$. If $\{P_{i_1},P_{i_2},\cdots,P_{i_k}\}$ is a set of isolated prime divisors of $n$ then $q_{i_1} \wedge q_{i_2} \wedge \cdots \wedge q_{i_k}$ depends only on this set and not on the particular reduced primary decomposition of $n$.
\end{theorem}
\begin{proof}
Let $S = R \backslash (P_{i_1} \cup P_{i_2} \cup \cdots \cup P_{i_k})$. Then $S$ is a multiplicatively closed set in $R$. To prove the theorem it is sufficient to show $n_{S} = q_{i_1} \wedge q_{i_2} \wedge \cdots \wedge q_{i_k}$. For this we only need to show that $S \cap P_{i} = \emptyset$, for $i=i_{1},i_{2},\cdots,i_{k}$, and $S \cap P_{i} \neq \emptyset$, for $i \neq i_{1},i_{2},\cdots,i_{k}$, rest will follow from the Theorem \ref{S component}. The former is certainly true. Let $i \neq i_{1},i_{2},\cdots,i_{k}$. Since $\{P_{i_1},P_{i_2},\cdots,P_{i_k}\}$ is a set of isolated prime divisors of $n$, we have $P_{i} \nsubseteq P_{i_j}$, for $j=1,2,\cdots,k$, and hence $P_{i} \nsubseteq \cup_{j=1}^{k} P_{i_j}$. Thus $S \cap P_{i} \neq \emptyset$.
\end{proof}

The following consequence on the uniqueness of primary components corresponding to minimal prime divisors of a submodule element is immediate from Theorem \ref{S component} and Theorem \ref{isolated set of prime divisors}.
\begin{corollary}
Suppose that $n$ be a submodule element of an le-module $M$ and $P$ be a minimal prime divisor of $n$. If a $P$-primary element $q$ occurs in a reduced primary decomposition of $n$, then $q$ occurs in every reduced primary decomposition of $n$.
\end{corollary}

\bibliographystyle{amsplain}

\begin{thebibliography}{10}
\baselineskip 5mm

\bibitem{Al-Khouja}
E. A. Al-Khouja, Maximal elements and prime elements in lattice modules, \emph{Damascus University Journal for BASIC SCIENCES}, \textbf{19(2)}, 2003, 9-21.

\bibitem{Anderson}
D. D. Anderson, \emph{Multiplicative lattices}, Dissertation, University of Chicago, 1974.

\bibitem{Anderson 2}
D. D. Anderson, Abstract commutative ideal theory without chain condition, \emph{Algebra Universalis}, \textbf{6}, 1976, 131-145.

\bibitem{Anderson and Johnson}
D. D. Anderson and E. W. Johnson, Abstract ideal theory from Krull to the present, in: Ideal theoretic methods in commutative algebra(Columbia, MO, 1999), Lecture notes in Pure and Appl. Math, \textbf{220}, Marcel Dekkar, New York, 2001, 27-47.

\bibitem{Atiyah-MacDonald}
M. F. Atiyah and I. G. MacDonald, \emph{Introduction to commutative algebra}, Addison-Wesley Publishing Company, 1969.

\bibitem{Dilworth}
R. P. Dilworth, Abstract commutative ideal theory, \emph{Pacific J. Math}, \textbf{12}, 1962, 481-498.

\bibitem{Fuchs and Reis}
L. Fuchs and R. Reis, On lattice-ordered commutative semigroups, \emph{Algebra Univers.}, \textbf{50}, 2003, 341-357.

\bibitem{Gilmer and Heinzer}
R. Gilmer and W. Heinzer, The Laskerian property, power series rings and Noetherian spectra, \emph{Proc. Amer. Math. Sot.}, \textbf{79}, 1980, 13-16.

\bibitem{Heinzer and Lantz}
W. Heinzer and D. Lantz, The Laskerian property in commutative rings, \emph{J. Algebra}, \textbf{72}, 1981, 101–114.

\bibitem{Heinzer and Ohm}
W. Heinzer and J. Ohm, On the Noetherian-like rings of E.G. Evans. \emph{Proc. Amer. Math. Soc.}, \textbf{34}, 1972, 73–74.

\bibitem{Jayaram and Tekir}
C. Jayaram and U. Tekir, Q-modules, \emph{Turk J Math}, \textbf{33}, 2009, 215 – 225.

\bibitem{Johnson 1}
J. A. Johnson, a-adic completion of Noetherian lattice module, \emph{Fund Math}, \textbf{66}, 1970, 341-371.

\bibitem{Johnson 2}
J. A. Johnson, Quotients in Noetherian lattice modules, \emph{Proceedings of the American Mathematical Society}, \textbf{28(1)}, 1971, 71-74.

\bibitem{Johnson 3}
J. A. Johnson, Noetherian lattice modules and semi-local comletions, \emph{Fundamenta Mathematicae}, \textbf{73}, 1971-1972, 93-103.

\bibitem{Johnson and Johnson}
E. W. Johnson and J. A. Johnson, Lattice Modules over principal element domains, \emph{Comm. in Algebra}, \textbf{31 (7)} , 2003, 3505 - 3518.

\bibitem{Krull}
W. Krull, Axiomatische begrundung der allgemeinen idealtheorie, \emph{Sitzungsber. phys.-med. Soz. Erlangen} \textbf{56}, 1924, 47-63.

\bibitem{Larsen-McCarthy}
M. D. Larsen and P. J. McCarthy, \emph{Multiplicative theory of ideals}, Volume-43, Academic Press, 1971.

\bibitem{CPLu}
C. P. Lu, Prime submodules of modules, \emph{Comment. Math. Univ. St. Paul.},  Vol. 22 No. 1, 1984, 61-69.

\bibitem{Majcherek}
D. J. Majcherek, \emph{Multiplicative lattice versions of some results from Noetherian commutative rings}, Dissertation, University of California, Riverside, August 2014.

\bibitem{Manjarekar and Bingi}
C. S. Manjarekar and A. V. Bingi, Absorbing elements in lattice modules, \emph{International Electronic Journal of Algebra}, \textbf{19}, 2016, 58-76.

\bibitem{Manjarekar-Kandale}
C. S. Manjarekar and U. N. Kandale, Primary decomposition in lattice modules, \emph{European Journal of Pure and Applied Mathematics}, \textbf{7}, 2014, 201-209.

\bibitem{McCasland and Moore}
R. L. McCasland and M. E. Moore, Prime submodules, \emph{Communications in Algebra}, \textbf{20:6}, 1992, 1803-1817.

\bibitem{Nakkar}
H. M. Nakkar, Localization in multiplicative lattice modules, (Russian), \emph{Mat. Issled}, \textbf{2(32)}, 1974, 88 - 108.

\bibitem{Nakkar and Anderson}
H. M. Nakkar and D. D. Anderson, Associated and weakly associated prime elements and primary decomposition in lattice modules, \emph{Algebra Universalis}, \textbf{25}, 1988, 196-209.

\bibitem{Radu}
N. Radu, Sur les anneaux coherents Laskeriens, \emph{Rev. Roum. Math. Pures Appl}, \textbf{11}, 1966, 865-867.

\bibitem{Smith}
P. F. Smith, Uniqueness of primary decompositions, \emph{Turk J Math}, \textbf{27}, 2003, 425-434.

\bibitem{Swanson}
Irena Swanson, \emph{Primary decompositions}, lecture notes.

\bibitem{Visweswaran}
S. Visweswaran, Laskerian pairs, \emph{J. Pure Appl. Algebra}, \textbf{59}, 1989, 87–110.

\bibitem{Ward and Dilworth 1}
M. Ward and R.P. Dilworth, Residuated lattices, \emph{Trans. Amer. Math. Soc}, \textbf{35}, 1939, 335-354.

\bibitem{Ward and Dilworth 2}
M. Ward and R.P. Dilworth, Lattice theory of ova, \emph{Ann. of Math. Soc}, \textbf{40}, 1939, 600-608.

\bibitem{Ward and Dilworth 3}
M. Ward and R.P. Dilworth, Evaluations over residuated structures, \emph{Ann. of Math. Soc}, \textbf{40}, 1939, 328-338.

\bibitem{Whitman 2}
D. G. Whitman, Ring theoretic lattices modules and the Hilbert polynomial, Dissertation, University of California Riverside, March 1967.

\bibitem{Whitman 1}
D. G. Whitman, On ring theoretic lattice modules, \emph{Fund. Math.}, \textbf{70}, 1971, 221-229.


\end{thebibliography}

\end{document}